\documentclass[12pt]{article}
\pdfoutput=1 
\usepackage{amsmath, dsfont, amssymb, amscd, caption,  color}
\usepackage{graphicx,subfigure}
\usepackage[margin=1in]{geometry}

\usepackage{times,soul}
\usepackage[square,sort,comma,numbers]{natbib}

\usepackage{setspace}
\usepackage{tikz}
\usepackage{url}
\usepackage{tikz-cd}

\usepackage{amsfonts}
\usepackage{fancybox}
\usepackage{algorithmic}
\usepackage{algorithm}
\usepackage{comment}
\usepackage{color}
\usepackage[colorlinks,citecolor=blue,urlcolor=blue]{hyperref}

\usepackage{paralist}
\usepackage{tikz,enumerate}
\usepackage{multirow,ctable}
\usepackage[normalem]{ulem}
\usepackage[framemethod=TikZ]{mdframed}
\usetikzlibrary{arrows}
\newdimen\arrowsize
\pgfarrowsdeclare{arcsq}{arcsq}
{
  \arrowsize=0.2pt
  \advance\arrowsize by .5\pgflinewidth
  \pgfarrowsleftextend{-4\arrowsize-.5\pgflinewidth}
  \pgfarrowsrightextend{.5\pgflinewidth}
}
{
  \arrowsize=1.5pt
  \advance\arrowsize by .5\pgflinewidth
  \pgfsetdash{}{0pt} 
  \pgfsetroundjoin   
  \pgfsetroundcap    
  \pgfpathmoveto{\pgfpoint{0\arrowsize}{0\arrowsize}}
  \pgfpatharc{-90}{-140}{4\arrowsize}
  \pgfusepathqstroke
  \pgfpathmoveto{\pgfpointorigin}
  \pgfpatharc{90}{140}{4\arrowsize}
  \pgfusepathqstroke
}

\newcommand{\bpm}{\begin{pmatrix}}
\newcommand{\epm}{\end{pmatrix}}

\graphicspath{{images/}}

\newtheorem{theorem}{Theorem}
\newtheorem{lemma}{Lemma}
\newtheorem{prop}{Proposition}
\newtheorem{corollary}{Corollary}

\newtheorem{remark}{Remark}

\newtheorem{definition}{Definition}

\newenvironment{proof}[1][. ]{{\bf Proof#1}}{\hfill$\square$\vskip\baselineskip}

\begin{document}

\title{SYMMETRIES OF EINSTEIN-WEYL MANIFOLDS WITH BOUNDARY}
\author{Rouzbeh Mohseni}

\maketitle

\begin{abstract}
Starting from a real analytic surface $\mathcal{M}$ with a real analytic conformal Cartan connection A. Bor\'owka constructed a minitwistor space of an asymptotically hyperbolic Einstein-Weyl manifold with $\mathcal{M}$ being the boundary. In this article, starting from a symmetry of conformal Cartan connection, we prove that symmetries of conformal Cartan connection on $\mathcal{M}$ can be extended to symmetries of the obtained Einstein-Weyl manifold.


\newpage 

\end{abstract}

\section{Introduction}
A complex manifold $M$ with a conformal structure $[g]$ is a Weyl manifold if it is equipped with a holomorphic connection $\mathcal{D}$ that preserves $[g]$. Furthermore, it is called an Einstein-Weyl manifold if the symmetric trace-free part of the Ricci tensor of $\mathcal{D}$ vanishes. In \cite{5} N. Hitchin introduced a twistor correspondence for 3-dimensional Einstein-Weyl manifolds which is called minitwistor correspondence or Hitchin correspondence. Moreover, in \cite{7} P. Jones and K. Tod gave a relation between the work of R. Penrose on twistor spaces \cite{9} and the Hitchin correspondence. 

A. Bor\'owka \cite{2} starting from a real analytic surface $\mathcal{M}$ with a real analytic conformal Cartan connection constructed a complex surface and then proved that the constructed surface is, in fact, a minitwistor space of an asymptotically hyperbolic Einstein-Weyl space with $\mathcal{M}$ being the boundary. A. Bor\'owka also gave a description on how this fits with work of Jones and Tod by explicitly realizing the minitwistor space as a quotient of a twistor space with a local $\mathbb{C}^{\times}$ action and $\mathcal{M}$ being the fixed point set.

After having a correspondence or a construction, it is natural to ask what type of initial data is kept through the construction and in the case of correspondence, what will the initial data relate to in the corresponding space. One of the examples, is the work done by A. Bor\'owka and H. Winther \cite{1}, in which they investigate the symmetries in case of the generalized Feix-Kaledin construction. In this article, we want to do a similar investigation for the construction done in \cite{2}, therefore, starting from a symmetry of conformal Cartan connection on a complex surface (see Definition \ref{definition 3}), we try to determine sufficient conditions for which a symmetry of the Cartan connection gives the symmetry of the Einstein-Weyl structure. The final result is that, under some mild conditions on the bundle appearing in the definition of the conformal Cartan connection, the symmetry we start with on the boundary $\mathcal{M}$ can be extended to a symmetry of the minitwistor space and therefore, it induces a symmetry of the corresponding Einstein-Weyl manifold. This result is analogous to the result in \cite{1} where the c-projective symmetries under given conditions extend from the fixed points set of a circle action to quaternionic symmetries.

In Section \ref{section 2}, we review necessary background needed for construction, then in \textsection{\ref{subsection 2.2}}, we review the construction done in \cite{2}. In Section \ref{section 3}, we follow the construction and show how the symmetry carries through the construction and finally, we obtain a symmetry of the minitwistor space. In Section \ref{Section 4}, using this result together with the result from \textsection{\ref{subsection 2.1}} we show that the symmetry we started with on the boundary can be extended to a symmetry of the Einstein-Weyl space.

\section{Background}{\label{section 2}}
\textbf{Complexification:}
For any n-dimensional real-analytic manifold $\mathcal{M}$, its complexification $\mathcal{M^{\mathbb{C}}}$ is a holomorphic manifold that contains $\mathcal{M}$ as a fixed point set of the real structure (i.e. an anti-holomorphic involution) and $dim_{\mathbb{C}}\mathcal{M^{\mathbb{C}}}= dim_{\mathbb{R}}\mathcal{M}$. $\mathcal{M^{\mathbb{C}}}$ can be constructed by using holomorphic extensions of the real-analytic transition functions on $\mathcal{M}$ and the real structure will be given by the complex conjugation. Similarly, using holomorphic extensions, real-analytic objects like  functions, bundles and connections on $\mathcal{M}$ can be extended to $\mathcal{M^{\mathbb{C}}}$ .\newpage

\flushleft\textbf{Hitchin correspondence:}

\begin{definition}
Let $(\mathcal{M},[g])$ be a conformal manifold with a compatible torsion-free connection $\mathcal{D}$ (i.e. a Weyl connection). Then $(\mathcal{M},[g],\mathcal{D})$ is called an Einstein-Weyl manifold if the symmetric trace-free part of the Ricci tensor of $\mathcal{D}$ vanishes. (For more informations about Einstein-Weyl manifolds see \cite{4},\cite{6},\cite{7} and \cite{8}).

\end{definition}

In 1967 \cite{9} R. Penrose proposed twistor theory as a possibility for quantizing space-time and fields and then in 1976 \cite{10} gave a description for curved twistor theory. Later in 1982 N. Hitchin obtained a similar construction to Penrose  and provided a one-to-one correspondence between the three-dimensional Einstein-Weyl spaces and minitwistor space and in particular, proved the following theorem:

\begin{theorem}
(Hitchin\cite{5}, see \cite{2}). Let $T$ be a surface such that:

1. There is a family of non-singular holomorphic projective lines $\mathbb{CP}^{1}$ each with normal bundle isomorphic to $\mathcal{O}(2)$, called minitwistor lines.

2. $T$ has a real structure, which induces the antipodal map of $\mathbb{CP}^{1}$ on lines from the family that are invariant under this real structure.

Then the parameter space of projective lines invariant under the real structure is an Einstein-Weyl manifold.
\end{theorem}

In 1985 \cite{7} P. Jones and K. Tod related the work of Penrose to the Hitchin correspondence by realizing the spaces in the Hitchin correspondence as the quotient space of the spaces in the Penrose correspondence by a conformal Killing vector field and a holomorphic vector field respectively. Moreover, they proved that the spaces in the Penrose correspondence can be constructed from the quotient spaces provided that the Einstein-Weyl space is equipped with so-called abelian monopole.\newline

\textbf{Conformal Cartan connection:}(see \cite{2}, \cite{3})

\begin{definition}\label{definition 2}
A conformal Cartan connection on an n-manifold $\Sigma$ is a quadruple ($V$, $\langle.,.\rangle$, $\Lambda$, $\mathcal{D}$) where:

$\blacklozenge$ $V$ is a rank $n+2$ vector bundle with inner product $\langle.,.\rangle$ over $\Sigma$,

$\blacklozenge$ $\Lambda \subset V$ is a null line subbundle over $\Sigma$,

$\blacklozenge$ $\mathcal{D}$ is a linear metric connection in the vector bundle V satisfying the Cartan condition, i.e. $\epsilon:=\mathcal{D}\mid_{\Lambda}$ mod $\Lambda$ is an isomorphism from $T\Sigma \otimes \Lambda$ to $\Lambda^{\perp}/\Lambda$.
\end{definition}

For the purposes of this article we will restrict to the case of $n=2$. Note that in this dimension the conformal structure does not fully determine the conformal Cartan connection, however, in higher dimensions the relation is bidirectional, which will not be discussed here.\newline

Let $\Sigma$ be a complex surface  with a complex Cartan connection ($V$, $\langle.,.\rangle$, $\Lambda$, $\mathcal{D}$) given as a complexification of a real analytic surface $\mathcal{M}$ with a Cartan connection. Suppose that the fiber bundle $V$ is the associated bundle to the tangent bundle $T\Sigma$. Let $Z$ be a vector field on $\mathcal{M}$ which can then be complexified to obtain a vector field $X$ on $\Sigma$ and $X$ generates a local 1-parameter group of transformations $\phi_t$. 

\begin{definition} \label{definition 3}
The Cartan connection is preserved by $\phi_t$ if and only if the following holds for sufficiently small values of $t$:
\\
1. $\phi_t$ preserves the line bundle $\Lambda$.\\
2. $\phi_t$ preserves the inner product, i.e. the following holds:

$\langle Y,Z \rangle=\langle(\phi_t)_*Y,(\phi_t)_*Z\rangle.  
\; \forall \; Y,Z\in\Gamma(T\Sigma)$\\
3. $(\phi_t)_\ast \mathcal{D}=\mathcal{D}$, i.e. the connection is preserved.
\\
In this case, the vector field $X$ is called a symmetry of conformal Cartan connection.

\end{definition}
\begin{remark}
The isomorphism $\epsilon$ is also preserved by the symmetry, since it preserves the connection $\mathcal{D}$ and the line bundle $\Lambda$.
\end{remark}

\subsection{Symmetries of minitwistor spaces}{\label{subsection 2.1}}

Minitwistor spaces and Einstein-Weyl 3-dimensional manifolds were related by the work of Hitchin and this correspondence is called the minitwistor correspondence or Hitchin correspondece. In this section, we discuss the properties and relationship of a particular kind of symmetry of these spaces.

\begin{definition}
Let $M$ be a complex 3-dimensional manifold with a Weyl structure $([g],\mathcal{D})$. A null plane is a 2-dimensional subspace $U$ of $T_w M$ for each point $w \in M$ such that $[g]$ degenerates on U.
\end{definition}

\begin{definition}
A null surface is a 2-dimensional submanifold $S\subset M$ such that for every $w\in M$, $T_{w} S$ is a null plane. 
\end{definition}

The parameter space of all minitwistor lines is a complex 3-dimensional manifold $M^{\mathbb{C}}$ with the real structure induced by the real structure on $T$ and the real submanifold $M$ is given as the parameter space of minitwistor lines invariant under the real structure; $M^{\mathbb{C}}$ is a complexification of $M$. Conversely, $T$ can be defined as the space of totally geodesic null hypersurfaces in $M$. As points $l \in M$ correspond to real minitwistor lines in $T$, therefore for points $w\in T$ it makes sense to consider $w \in l$ . There exists two families of submanifolds defined as $M^{\mathbb{C}}_w:=\{l \in M| w \in l\}$, which is a 2-dimensional complex submanifold and $M^{\mathbb{C}}_{w,w^{\prime}}:=\{l \in M| w,w^{\prime} \in l\}$, which is 1-dimensional. The complexified Einstein-Weyl structure on $M^{\mathbb{C}}$ can be determined as follows:

\begin{prop}{\label{Proposition 1}}(see \cite{6})
There exists a unique torsion free complexified Einstein-Weyl structure ($[g]$,$\mathcal{D}$) on $M^{\mathbb{C}}$ that satisfies:
\begin{enumerate}
    \item The family $\{M^{\mathbb{C}}_w\}_{w\in T}$ and the set of null surfaces of $[g]$ coincide.
    \item The family $\{M^{\mathbb{C}}_{w,w^{\prime}}\}_{w,w^{\prime}\in T}$ and the set of geodesics coincide.
    \item A curve $M^{\mathbb{C}}_{w,w^{\prime}}$ is null geodesics if and only if $w$ is a double point in $l$.
\end{enumerate}
\end{prop}

\begin{definition}
Let ($M$,$[g]$,$\mathcal{D}$) be an Einstein-Weyl manifold. A diffeomorphism is called a symmetry of ($M$,$[g]$,$\mathcal{D}$) if it preserves the Einstein-Weyl structure ($[g]$,$\mathcal{D}$).
\end{definition}
\newpage

Let $X$ be a holomorphic vector field on $T$ and $\phi_t$ the transformation induced by $X$.

\begin{lemma}{\label{Lemma 1}}
 The transformation $\phi_t$ preserves the minitwistor lines.
 
 \begin{proof}
 \normalfont
 $\phi_t$ for sufficiently small $t$ preserves the normal bundle $\mathcal{O}(2)$, therefore the minitwistor lines are preserved.
 \end{proof}
 \end{lemma}
 If the vector field is real then it, moreover, preserves real minitwistor lines. The following theorem is a common knowledge by experts in the field but since we were unable to find a source for it, we will state and prove it here.

\begin{theorem}{\label{Theorem 2}}
Real holomorphic vector fields on $T$ correspond to symmetries of ($M$,$[g]$,$\mathcal{D}$).
\end{theorem}

\begin{proof}
\normalfont
 First we want to show that the families of submanifolds $\{M_w\}$ and $\{M_{w,w^{\prime}}\}$ are preserved by the transformation $\phi_t$. Observe that under the action of the transformation $\phi_t:w\mapsto \phi_t(w)$, $w \in M$, the submanifold $M_w$ are transformed into $M_{\phi_t(w)} :=\{l\in M|\phi_t(w) \in l\}$. Furthermore, note that the twistor lines containing $w$ are the transformed into the twistor lines containing $\phi_t(w)$, hence by Lemma \ref{Lemma 1},  $\phi_t (M_w)\subseteq M_{\phi_t(w)}$. Moreover, since $\phi_t$ is an isomorphism the inclusion in the other direction is obtained by the inverse, therefore, the family of submanifolds $\{M_w\}$ are preserved. The proof for $\{M_{w,w^{\prime}}\}$ is similar, however, it is worth to notice that it is necessary for the points $w$ and $w^\prime$ to be sufficiently near each other. Hence, by Proposition \ref{Proposition 1} holomorphic vector fields on $T$ correspond to symmetries of the underlying complexified Einstein-Weyl manifold. The reality condition on the vector fields imply that the symmetries restrict to symmetries of the underlying real Einstein-Weyl manifold.
\end{proof}

\subsection{\textbf{Review of the twistor construction:}}{\label{subsection 2.2}}

In \cite{2} A. Bor\'owka gives a description for a construction of minitwistor spaces for asymptotically hyperbolic Einstein-Weyl spaces. In this section, a concise review of the first part of the construction is given, for more details and proofs see \cite{2}.\newline

Let $\mathcal{M}$ be a real analytic surface with a Cartan connection, by complexification we obtain a complex surface $\Sigma$ with a complexified conformal Cartan connection ($V$, $\langle.,.\rangle$, $\Lambda$, $\mathcal{D}$) defined as in Definition \ref{definition 2}. Moreover, there exists $\Lambda^{0}\subset V$ that is the annihilator of $\Lambda$ and for each point $\sigma \in \Sigma$, we will have two null planes $U^{+}_{\sigma}\subset \Lambda^{0}$ and $U^{-}_{\sigma}\subset \Lambda^{0}$ , which are defined using the induced degenerated inner product on $\Lambda^{0}_{\sigma}$, as the solutions to $\langle a,a\rangle=0$ for $a\in \Lambda^{0}_{\sigma}$ . Let $U^+$ and $U^-$ be the two null subbundles of $\Lambda^\bot \subset V$ defined fiberwise by these null planes with $\Lambda=U^+\cap U^-$.
\newline

Using the isomorphism $\epsilon$ between $T\Sigma \otimes \Lambda$ and $\Lambda^\bot/\Lambda$  given in Definition \ref{definition 2}, two line subbundles $t^+,t^-$ of the tangent bundle $T\Sigma$ can be defined as follows

\begin{align}
\epsilon(t^+\otimes \Lambda)=U^+/\Lambda \;\;\;\;\;\;\;\; \texttt{and}\;\;\; \epsilon(t^-\otimes \Lambda )=U^-/\Lambda
\end{align}

These define two families of curves $C^+$ and $C^-$ as integral curves of the line subbundles $t^+$ and $t^-$ respectively. Moreover, fiber bundle $F^+$(respectively $F^-$) with fibers given by $F^+_{\sigma}:=\mathbb{P}\left(U^+_{\sigma}\right)$ [respectively $F^-_{\sigma}:=\mathbb{P}\left(U^-_{\sigma}\right)$] is defined.

The connection $\mathcal{D}$ on the base manifold induces a connection along the curves $C^+$(respectively $C^-$) and using this connection it is possible to horizontally lift the curves from $C^+$(respectively $C^-$) to $F^+$(respectively $F^-$). Furthermore, $\Sigma$ can be restricted in such a way that the curves from each family do not intersect each other.

\begin{prop}
The horizontally lifted curves from $C^+$(respectively $C^-$) families, locally foliate the total space of the bundle $F^+$(respectively $F^-$) and the leaf space of the foliations is a manifold which is denoted by $T^+$(respectively $T^-$).

\begin{proof}
\normalfont
See \cite{2}
\end{proof}
\end{prop}
We restrict the manifold $\Sigma$ such that any horizontally lifted curve from the $C^+$ family intersects a horizontally lifted curve from $C^-$ family at most once and if we denote by $\sigma \in \Sigma$ the points of the intersection, then $\mathbb{P}(\Lambda)_\sigma=F^+_\sigma \cap F^-_\sigma$ holds. Therefore, we have for any point $\mathbb{P}(\Lambda)_\sigma$ exactly one element of $T^+$ which intersects one element $T^-$ and it enables us to glue the leaf spaces $T^+$ and $T^-$.

\begin{definition}
$T^+$ and $T^-$ can be glued together in the following way:
\begin{align}
T:=T^+\bigsqcup_{\thicksim} T^-,
\end{align}
 $ \forall t^+ \in T^+$, $ t^- \in T^-$, $ t^+ \thicksim t^- \Leftrightarrow \exists \, \sigma \in \Sigma $ \, : $t^+ \cap t^- = \lbrace\mathbb{P}(\Lambda)_\sigma\rbrace.$ 
\end{definition}

\begin{definition}
For each curve $c^+ \in C^+$ the family of its horizontal lifts define a projective line in $T^+$ which will be denoted by $l^{+}_{c^+}$ and analogously $l^{-}_{c^-}$ in $T^-$ is defined for $c^- \in C^-$.
\end{definition}

Since for each point $\sigma \in \Sigma$ there exists exactly one curve from each family such that $\sigma \in c^{\pm}$, we will use the notation $l^{+}_{\sigma}:= l^{+}_{c^+}$ and $l^{-}_{\sigma}:= l^{-}_{c^-}$ instead. Also note that the minitwistor lines are deformations of this line pairs. $T$ is a minitwistor space of an Einstein-Weyl manifold and it admits a real structure induced naturally by the initial complexification of the Cartan connection. Furthermore, the pairs of intersecting lines $l^{\pm}_{\sigma}$ correspond to points on a boundary of this Einstein-Weyl manifold.

\section{Construction of a vector field on the minitwistor space}{\label{section 3}}
 Let $\Sigma$ be a complex surface with a complexified conformal Cartan connection, which is a complexification of a real analytic surface $\mathcal{M}$ and X be a holomorphic vector field on $\Sigma$ which is obtained as a complexification of a symmetry of the conformal Cartan connection (see Definition \ref{definition 3}) as described in Section \ref{section 2}. $T$ is a minitwistor space as in \textsection{\ref{subsection 2.2}} and $\phi_t$ is the flow of the vector field $X$. In this section, we argue how $X$ induces a transformation on the minitwistor space $T$. Then in the next section we will show that it is in fact a symmetry of the minitwistor space and therefore it induces a symmetry of the underlying Einstein-Weyl manifold.

\begin{lemma}{\label{Lemma 2}}
 The fiber bundles $U^+$ and $U^-$ are preserved by $\phi_t$.
 \\
 \begin{proof}
 \normalfont
  The transformation induced by $\phi_t$ preserves the inner product, therefore it preserves the fiber bundles $U^+$ and $U^-$.
 \end{proof}
 \end{lemma}

\begin{lemma}{\label{Lemma 3}}
 Let $t^{\pm}$ be the line subbundles of $T\Sigma$ defined as in \textsection{2.2}. $t^{\pm}$ are preserved by $\phi_t$ and hence the families of curves $c^\pm$, which are the integral curves of $t^\pm$ are also preserved.
 
 \begin{proof}
 \normalfont
 The line bundles $t^{\pm}$ were defined using the isomorphism $\epsilon$, the fiber bundles $U^{\pm}$ and the line bundle $\Lambda$, since $\phi_t$ preserves all of them, $t^{\pm}$ will also be preserved. 
 \end{proof}
 \end{lemma}
 
 Recall that the vector bundle $V$ is an associated bundle to the tangent bundle $T\Sigma$, hence there exists a transformation $\Tilde{\phi}_t$ on $V$, which is induced by $\phi_t$ as follows:
 
 \begin{align}
 \tilde{\phi}_t \left(\sigma,v\right) = \left( \phi_t(\sigma),(\phi_{t})_*(v)\right).\label{1}
 \end{align}

 \begin{lemma}{\label{Lemma 4}}
 The transformation $\Tilde{\phi}_t$ (\ref{1}) preserves the fiber bundles $U^\pm$.
 
 \begin{proof}
 \normalfont
  The fiber bundles $U^\pm$ are subbundles of $V$. We abuse the notation and denote maps $\Tilde{\phi}_{t}|_{U^{\pm}}$ also by $\Tilde{\phi}_t$ .We want to show that the following diagram exists:

  \begin{displaymath}
 \begin{tikzcd}
  U^\pm \arrow[r,"\tilde{\phi_t}"] \arrow[d,"\pi "]
    & U^\pm \arrow[d,"\pi "]  \\
    \Sigma  \arrow[r, "\phi_t "]
    &  \Sigma    
 \end{tikzcd}
 \end{displaymath}
 
 Since we proved in Lemma \ref{Lemma 2} that $\phi_t$ preserves the fiber bundles $U^\pm$, $\tilde{\phi}_t$ will preserve the fiber bundles $U^\pm$, therefore, the diagram is well-defined. 

 \end{proof}
 \end{lemma}
 
  It was discussed in \textsection{\ref{subsection 2.2}} that there exist two families of curves $C^+$ and $C^-$ on the base manifold $\Sigma$ which are then lifted to the fiber bundles $F^+$ and $F^-$ respectively and the lifted curves foliate the fiber bundles $F^{\pm}$. Now in order to show that the symmetry descends to  $T^\pm$ , first we have to prove that it preserves the lifted curves.
  
 \begin{lemma}{\label{Lemma 5}}
 The transformation $\tilde{\phi_{t}}$ preserves the lifted curves of $C^\pm$ families and therefore gives a transformation on the leaf spaces of the curves lifted to $U^\pm$.
 
 \begin{proof}
 \normalfont
  Let $C^\pm$ and $ \tilde{C}^\pm$ denote respectively the family of curves on $\Sigma$ and the families of curves lifted to the fiber bundles $U^\pm$. Take a curve $c^{+}_{1} \in C^+$ and let $\tilde{c}^{+}_{1}\in \Tilde{C}^+$ be a curve obtained by horizontally lifting $c^{+}_{1}$. By Lemma \ref{Lemma 3}, $\phi_t$ maps the curve $c^{+}_{1}$ to another curve $c^{+}_{2}$ from the same family. We want to show that $\tilde{\phi}_t$ transforms the curve $\tilde{c}^{+}_{1}$ into a horizontal lift of the curve $c_2^+$. Therefore, this will imply the image of $\Tilde{c}^+_1$ belongs to $\Tilde{C}^+$ and hence we will have the following diagram:
  
 \begin{displaymath}
  \begin{tikzcd}
   \tilde{C^{\pm}} \arrow[r, "\tilde{\phi_t} "] \arrow[d, "\pi "]
   & \tilde{C^{\pm}} \arrow[d, "\pi "]\\
   C^\pm \arrow[r, "\phi_t "]
   & C^\pm
 \end{tikzcd}
 \end{displaymath}
 
  Let $X_{\tilde{c}^{+}_{1}}$, $X_{c_1^+}$ and $X_{c_2^+}$ be tangent vector fields to the curves  $\tilde{c}^{+}_{1}$, $c_1^+$ and $c_2^+$ respectively and we denote $X_2:=(\tilde{\phi}_t)_{*}(X_{\tilde{c}_{1}^{+}})$. The fact that $X_2$ is a horizontal lift of $X_{c^{+}_{2}}$ comes from definition of $\tilde{\phi}_t$ (see Equation \ref{1}): we have that $(\phi_t)_\ast \mathcal{D}=\mathcal{D}$ holds, hence, this follows from the properties of the horizontal lift. The proof for curves of $C^-$ family is analogous.
 
 \end{proof}
 \end{lemma}

  \begin{prop}
  {\label{Proposition 3}}
 The constructed flow on the leaf space of curves lifted to $U^\pm$ is $\mathds{C}^{\times}$ invariant.
 \\
 \begin{proof}
 \normalfont
 The $\mathbb{C}^{\times}$ action maps isomorphically curves to curves and for the constructed flow $\tilde{\phi}_t$ (\ref{1}) the following holds:
 \begin{align}
 \tilde{\phi_{t}}\left( \sigma,u^\pm\right)= \left( \phi_t(\sigma),\phi_{t*}(\lambda u^\pm)\right)= \left( \phi_t(\sigma),\lambda\phi_{t*}(u^\pm)\right),\;\;\;\;\;\;\; \forall \lambda \in \mathbb{C}, \,u^\pm \in U^\pm.
 \end{align}
 which is a result of $\phi_{t*}$ being a linear isomorphism. 
 \end{proof}
 \end{prop}
 
 In Section \ref{subsection 2.2}, two fiber bundles $F^+$ and $F^-$ were defined fiberwise by $F^{\pm}_{\sigma}:=\mathds{P}\left(U^{\pm}_{\sigma}\right)$   respectively, and as a result of Proposition \ref{Proposition 3} we obtain transformations $\tilde{\phi}_{t}^{\pm}$ on $F^\pm$. Furthermore, these bundles were foliated by lifted curves, and their leaf spaces were denoted by $T^\pm$, hence, by Lemma \ref{Lemma 5} we obtain the following corollary.
 
 \begin{corollary}
 The obtained transformations descend to transformations on $T^+$ and $T^-$ and are denoted by $\tilde{\phi}^{\prime +}_t$ and $\tilde{\phi}^{\prime -}_t$ respectively.
 \end{corollary}
 
  Recall that the minitwistor space $T$ is obtained by the gluing of $T^+$ and $T^-$ and we need to check that the transformations $\tilde{\phi}^{\prime +}_t$ and $\tilde{\phi}^{\prime -}_t$ coincide on the gluing part.

 \begin{prop}
 The transformations $\tilde{\phi}^{\prime +}_t$ and $\tilde{\phi}^{\prime -}_t$ are compatible with the gluing of $T^+$ and $T^-$, therefore, they induce a vector field on the minitwistor space $T$ denoted by $\phi_{t}^{\prime}$.
 \\
 \begin{proof}
 \normalfont
 The curves from the two families $C^+$ and $C^-$ may intersect each other at most in one point and this point lies in $\Lambda$ and the gluing is given by identifying curves that intersect each other in $\Lambda$. We proved in Lemma \ref{Lemma 5} that the flow maps the curves from each family to a curve which is also in that family of curves. What remains to prove is that the point of intersection is preserved, which is a consequence of the flow preserving $\Lambda$.
 \end{proof}
 \end{prop}

 \section{Properties of the vector field}\label{Section 4}
 
 Let $\tilde{X}$ be the vector field corresponding to the transformation $\tilde{\phi}^{\prime}_t$. In the previous section, we constructed a vector field on the minitwistor space arising from the construction in \cite{2}. Now by studying its properties, we will show that it gives a symmetry of the corresponding Einstein-Weyl space, which on the boundary $\mathcal{M}$ coincides with our initial symmetry of the conformal Cartan connection.

\begin{lemma}\label{Lemma 6}
 The transformation $\tilde{\phi}^{\prime}_{t}$ preserves the real structure on the minitwistor space.
 
 \begin{proof}
 \normalfont
 Recall that both the vector field $X$ and the real structure on the manifold $\Sigma$ were introduced by complexification from the underlying real manifold. As the real structure on the minitwistor space from \cite{2} was constructed using the real structure of this complexification, it is straightforward to show that the vector field $\tilde{X}$, which arises from $X$ preserves this real structure. 
 \end{proof}
\end{lemma}

\begin{theorem}{\label{90}}
The real holomorphic vector field $\tilde{X}$ on the minitwistor space corresponds to a symmetry $Y$ on the corresponding Einstein-Weyl manifold $M$.

\begin{proof}
\normalfont
In Section \ref{section 3}, we constructed the vector field $\tilde{X}$ and in Lemma \ref{Lemma 6}, we proved that it is in fact a real vector field. Therefore, as an immediate result of Theorem \ref{Theorem 2}, $\tilde{X}$ corresponds to a symmetry on $M$, which will be denoted by $Y$.
\end{proof}
\end{theorem}

Starting from the vector field $Z$ on real analytic surface $\mathcal{M}$, which is the symmetry of Cartan connection, by complexification we obtained a vector field $X$ on the complex surface $\Sigma$. Furthermore, we constructed the vector field $\tilde{X}$ that as was proved in Thereom \ref{90} corresponds to the vector field $Y$ on the Einstein-Weyl manifold $M$, which has $\mathcal{M}$ as its boundary. Now we are ready to state the main result of the paper.

\begin{theorem}\label{theorem 4}
Let $\mathcal{M}$ be a real analytic surface with a conformal Cartan connection and a symmetry $Z$. Suppose that the vector bundle $V$ used in Definition \ref{definition 2} is an associated bundle of the tangent bundle $T\mathcal{M}$ and let $M$ be an Einstein-Weyl manifold constructed from $\mathcal{M}$ via construction from \cite{2}. Then there exists a vector field $Y'$ on the manifold with boundary $M\cup \mathcal{M}$ such that

\begin{equation}
    Y'|_{\mathcal{M}}=Z
\end{equation}

\begin{proof}
\normalfont
We take $Y'|_M=Y$. The line pairs $l^{\pm}_\sigma$ are preserved by the vector field $\tilde{X}$, which is a result of Lemma \ref{Lemma 5}. These intersecting line pairs correspond to points on $\Sigma$ and the real ones to points on $\mathcal{M}$. As a result $\tilde{X}$ induces a transformation on $\mathcal{M}$, which by definition is equal to $Z$ since $\tilde{X}$ was constructed using the vector field $X$ that is a  complexification of $Z$.
\end{proof}
\end{theorem}

\subsection*{Acknowledgements}

I would like to thank Aleksandra Bor\'owka for helpful discussions and valuable comments. This work was supported by Grants N16/MNS/000001 and N16/DBS/000006 and I would like to thank Institute of Mathematics of Jagiellonian University in Krakow for financial support.

JAGIELLONIAN UNIVERSITY IN KRAKOW, INSTITUTE OF MATHEMATICS, 30-348 KRAKOW.

{\it{E-mail address}}: \textbf{rouzbeh.mohseni@doctoral.uj.edu.pl}

\end{document}